\documentclass[11pt,reqno]{amsart} 
 \usepackage{amsmath,amssymb,amsthm}
 \usepackage{graphicx}
 \usepackage{cite}
 \usepackage[mathscr]{eucal}

 \theoremstyle{plain}
 \newtheorem{theorem}{Theorem}[section]  
   
 \newtheorem{lemma}[theorem]{Lemma}  
 
 \newtheorem*{theorem*}{Theorem}

 \theoremstyle{plain}

 \newtheorem*{example*}{Example}

 \newtheoremstyle{citing}% name
   {3pt}%      Space above, empty = `usual value'
   {3pt}%      Space below
   {\itshape}% Body font
   {}%         Indent amount (empty = no indent, \parindent = para indent)
   {\bfseries}% Thm head font
   {.}%        Punctuation after thm head
   {.5em}%     Space after thm head: " " = normal interword space;
   {\thmnote{#3}}% Thm head spec

 \theoremstyle{citing}
\numberwithin{equation}{section}
\allowdisplaybreaks[1]

\newlength{\intwidth}
\makeatletter
\DeclareRobustCommand{\cpvint}[2]
    {\mathop{%
       \text{%
         \settowidth{\intwidth}{%
           \ifx\ilimits@\displaylimits
             $\int_{#1}^{#2}$%
           \else
             $\int$%
           \fi}%
         \makebox[0pt][l]{\makebox[\intwidth]{$\text{C}$}}%
         $\int_{#1}^{#2}$}}}
\makeatother

\makeatletter
\DeclareRobustCommand{\cpvintsmall}[2]
    {\mathop{%
       \text{%
         \settowidth{\intwidth}{%
           \ifx\ilimits@\displaylimits
             $\int_{#1}^{#2}$%
           \else
             $\int$%
           \fi}%
         \makebox[0pt][l]{\makebox[\intwidth]{$\text{{\tiny C}}$}}%
         $\int_{#1}^{#2}$}}}
\makeatother

\newcommand{\dist}{\text{dist}}

\newcommand{\rand}{\partial} 
\newcommand{\where}{:\:}

\newcommand{\laplace}{\Delta}

\newcommand{\nz}{{\mathbb N}}

\newcommand{\rz}{{\mathbb R}}  
\newcommand{\zz}{{\mathbb Z}}  
 
\newcommand{\eps}{\varepsilon}  
\renewcommand{\phi}{\varphi} 
\newcommand{\eval}{\vert}

\begin{document}
 
\title[Embedded constant curvature curves] {Embedded constant curvature curves on convex surfaces}

\author{Harold Rosenberg}
\address{Instituto de Matematica Pura y Aplicada\\
         110 Estrada Dona Castorina\\
         Rio de Janeiro 22460-320, Brazil
         }
\email{rosen@impa.br}

\author{Matthias Schneider}
\address{Ruprecht-Karls-Universit\"at\\
         Im Neuenheimer Feld 288\\
         69120 Heidelberg, Germany}
\email{mschneid@mathi.uni-heidelberg.de}

\date{March 16, 2011}  

\begin{abstract}
We prove the existence of embedded closed constant curvature curves
on convex surfaces.
\end{abstract}

\maketitle

\section{Introduction}
\label{sec:introduction}
Let $(S^2,g)$ be a two dimensional oriented sphere with a smooth Riemannian metric $g$.
We prove existence results for closed embedded curves with prescribed 
geodesic curvature in $(S^2,g)$, when the Gauss curvature $K_g$ of the metric
$g$ is positive.
In particular, we study the existence
of closed embedded constant curvature curves on strictly convex spheres.\\
Let $c:S^2 \to \rz$ be a smooth positive function.
We consider the following equation
for curves $\gamma$ on $S^2$:
\begin{align}
\label{eq:prescribed_geodesic}
D_{t,g} \dot \gamma(t) = |\dot \gamma(t)|_{g} c(\gamma(t)) J_{g}(\gamma(t)) \dot \gamma(t),
\end{align}
where $D_{t,g}$ is the covariant derivative with respect to $g$,
and $J_g(x)$ is the rotation by $\pi/2$ in $T_xS^2$ with respect to $g$ 
and the given orientation.
Solutions $\gamma$ to equation \eqref{eq:prescribed_geodesic} 
are constant speed curves with geodesic curvature $c_g(\gamma,t)$ given by 
$c(\gamma(t))$. We remark that, besides the geometric
interpretation, \eqref{eq:prescribed_geodesic} describes the motion
of a charged particle on $(S^2,g)$ in a magnetic field with magnetic form
$c \mu_g$, where $\mu_g$ denotes the volume form of $g$ 
(see \cite{MR890489,MR676612,MR1432462}).\\
From \cite{arXiv:0808.4038,MR1185286}
closed embedded solutions to \eqref{eq:prescribed_geodesic} exist, if the 
curvature function $c$ is large enough depending on the metric $g$.
If $g$ is 
$\frac14$-pinched, i.e. $\sup K_g < 4\inf K_g$,
then there are embedded closed solutions of \eqref{eq:prescribed_geodesic}
for every positive function $c$ (see \cite{arXiv:0808.4038,2010arXiv1010.1879R}).
It is conjectured in \cite[\S5]{MR676612}) and \cite{2010arXiv1010.1879R} 
that this remains
true for an arbitrary metric $g$ on $S^2$.
We note that, if $K_g$ and $c$ are positive, then 
from \cite{2010arXiv1010.1879R,robeday,arXiv:0903.1128} 
there are always Alexandrov embedded, closed solutions 
to \eqref{eq:prescribed_geodesic}, i.e. curves that bound an immersed disc.\\
We shall show that on strictly convex spheres, i.e. $K_g>0$, there are closed embedded solutions
to \eqref{eq:prescribed_geodesic}, if the curvature function is small enough depending
on the metric $g$. In particular, we show  
\begin{theorem}
Suppose $(S^2,g)$ has positive Gauss curvature.
Then there is $\eps_0>0$ such that for all $0< c \le \eps_0$
there are two embedded closed curves with constant geodesic curvature $c$.   
\end{theorem}
Hence, on strictly convex spheres there are closed, embedded constant curvature
curves for large and small values of $c>0$. We expect that this is true
for all $c>0$.\\
We use the degree theory developed in \cite{arXiv:0808.4038} to 
prove our existence result. The required 
compactness results are given in section \ref{sec:apriori-estimate}.
The a priori estimates follow from Reilly's formula. 
The fact that a geodesic cannot touch itself continues to hold 
for solutions to \eqref{eq:prescribed_geodesic} when the geodesic 
curvature is close to zero. This allows to carry out the
degree argument within the class of embedded curves.
The existence result is given in section \ref{sec:existence}.

\section{The apriori estimate}
\label{sec:apriori-estimate}
\begin{lemma}
\label{lem:length_bound}
Suppose $(S^2,g)$ has positive Gauss curvature $K_g$ and
$\gamma \in C^2(S^1,S^2)$ is an (Alexandrov) embedded curve
with nonnegative geodesic curvature. Then the length $L(\gamma)$ of $\gamma$ 
is bounded by
\begin{align*}
L(\gamma) \le 2\pi \sqrt{2} \big(\inf_{S^2} K_g\big)^{-\frac12}.  
\end{align*}
\end{lemma}
\begin{proof}
We use Reilly's formula \cite{MR0474149}: Let $(M,g)$ be a compact Riemannian manifold
with boundary $\rand M$, $f\in C^\infty(M)$, $z=f\eval_{\rand M}$ and 
$u=\frac{\rand f}{\rand n}$ on $\rand M$, where $n$ denotes the outer normal.
Then
\begin{align}
\label{eq:reillys}
\int_M (\bar{\laplace} f)^2 - |\bar{\nabla}^2 f|^2  =
&\int_M {\rm Ric}(\bar{\nabla} f, \bar{\nabla} f) \notag\\
&+
\int_{\rand M} (\laplace z + H u)u -\langle \nabla z,\nabla u\rangle
+ \Pi(\nabla z,\nabla z),   
\end{align}
where we denote by $\bar{\laplace}$, $\laplace$ and $\bar{\nabla}$, $\nabla$
the Laplacians and covariant derivatives on $M$ and $\rand M$ respectively;
$H$ is the mean curvature and $\Pi$ is the second fundamental form of $\rand M$.\\
If the curve $\gamma $ is embedded or Alexandrov embedded, then we may assume
that we are in the above situation with $\rand M=\gamma$.\\
We take $z$ an eigenfunction of $\lambda_1$ the first nontrivial eigenvalue
on $\rand M$
\begin{align*}
\laplace z + \lambda_1 z = 0 \text{ on } \rand M,   
\end{align*}
and $f$ its harmonic extension to $M$. In dimension two, \eqref{eq:reillys} 
leads to
\begin{align*}
\int_M (\bar{\laplace} f)^2 - |\bar{\nabla}^2 f|^2 =  
&\int_M K_g |\bar{\nabla} f|^2
+
\int_{\rand M} \laplace z u + c u^2 -\langle \nabla z,\nabla u\rangle
+ c |\nabla z|^2, 
\end{align*}
where $c$ is the geodesic curvature of $\rand M$ and $K_g$ denotes 
the Gauss curvature of $M$.
Using the fact that the geodesic curvature $c$ of $\rand M$ is nonnegative,
$f$ is harmonic, and $z$ is an eigenfunction, we obtain
\begin{align*}
0\ge  \big(\inf_{M} K_g\big) \int_M  |\bar{\nabla} f|^2 -2\lambda_1 \int_{\rand M} z u.   
\end{align*}
Integrating by parts again we see
\begin{align*}
\int_{\rand M} z u = \int_M  |\bar{\nabla} f|^2 + f\bar{\laplace} f. 
= \int_M  |\bar{\nabla} f|^2.   
\end{align*}
Since $z$ is a nontrivial eigenfunction, $f$ is non constant such that
we arrive at
\begin{align*}
\big(\inf_{M} K_g\big) \le 2 \lambda_1.  
\end{align*}
The first nontrivial eigenvalue $\lambda_1$ depends only on the length $L(\rand \Omega)$
of $\rand M$ and is given by
\begin{align*}
\lambda_1 = \frac{4\pi^2}{L(\rand \Omega)^2}.  
\end{align*}
This gives the claim.
\end{proof}

\begin{lemma}
\label{lem_geodesic_simple}
Let $(\gamma_n)$ be a sequence of simple closed curves 
converging in $C^{2}(S^{1},S^{2})$ to a non constant closed geodesic $\gamma$
in $(S^{2},g)$. Then $\gamma$ is simple as well.
\end{lemma}
\begin{proof}
To obtain a contradiction assume that there are $\theta_1\neq \theta_2$ in 
$S^1=\rz/\zz$
such that $\gamma(\theta_1)=\gamma(\theta_2)$. Since $\gamma$ is a limit of
simple curves and $|\dot\gamma|\equiv const$, there holds 
\begin{align*}
\dot\gamma(\theta_1) = \pm \dot\gamma(\theta_2).   
\end{align*}
From the uniqueness of geodesics we have for $t\in S^1$
\begin{align*}
\gamma(t)=\gamma(\pm(t-\theta_1)+\theta_2).  
\end{align*}
Setting $t=(\theta_1+\theta_2)/2)$ we find that
\begin{align*}
\gamma(t)=\gamma(t-\theta_1+\theta_2).    
\end{align*}
Consequently, $\gamma$ is a $n$-fold covering of a simple geodesic for
some $n\ge 2$. From the stability of the winding number, we get a
contradiction. 
\end{proof}
We denote by $g_{can}$ the 
standard round metric on $S^2$ with curvature $K_{g_{can}}\equiv 1$.
We fix a function $\phi \in C^\infty(S^2,\rz)$ and a conformal metric 
\begin{align*}
g= e^\phi g_{can}. 
\end{align*}
on $S^2$ with positive Gauss curvature $K_g>0$.
We consider the family of metrics $\{g_t\where t\in [0,1]\}$ 
defined by
\begin{align*}
g_t &:= e^{t\phi} g_{can}.
\end{align*}
Then the Gauss curvature $K_{g_t}$ of the metric $g_t$ satisfies
for some $K_0>0$
\begin{align*}
K_{g_t} &= e^{-t\phi}\big(-t\laplace_{g_{can}}(\phi)+2\big)\\
&= e^{-t\phi}\big(-t(2-K_{g}e^\phi)+2\big)\ge K_0,
\end{align*}
because $K_g$ is positive.

\begin{lemma}
\label{l:proper}
Suppose $c:S^{2}\to \rz$ is a nonnegative smooth
function. For $r\in [0,1]$ we define the set of closed curves $\mathcal{M}_r$ by    
\begin{align*}
\mathcal{M}_{r} := \{\gamma \in &C^{2}(S^{1},S^{2})\where \gamma \text{ is embedded, }
|\dot\gamma|_g \equiv {\rm const},\\
&\exists (t,s) \in [0,1] \times [0,r]:\: 
c_{g_t}(\gamma,\theta) = s c(\gamma(\theta)) \; \forall \theta \in S^{1}.\},  
\end{align*}
where $c_{g_t}(\gamma,\cdot)$ denotes the geodesic curvature of $\gamma$ with respect
to $g_t$.\\
Then there is $\eps_0>0$, such that $\mathcal{M}_{\eps_0}$ is compact.
Moreover, $\eps_0>0$ may be chosen uniformly with respect to $\|c\|_\infty$.
\end{lemma}
\begin{proof}
Let $(\gamma_n)_{n \in \nz}$ be a sequence in $\mathcal{M}_r$ for some $r>0$.
By Lemma \ref{lem:length_bound} and \eqref{eq:prescribed_geodesic}
we get a uniform bound in $C^{3}(S^{1},S^{2})$
and from the Gauss-Bonnet formula the length of $\gamma_n$ 
is bounded below, in both cases the bounds are uniform with respect to $r$.
Since the metrics $\{g_t\where t \in [0,1]\}$ are uniformly equivalent,
there is $C_0>0$ such that we have for all $t\in [0,1]$
\begin{align}
\label{eq:2}
|\dot\gamma_n|_{g_t}>(C_0)^{-1} \text{ and } \|\gamma_n\|_{C^{3}(S^{1},S^{2}),g_t} < C_0. 
\end{align}
Up to a subsequence we may assume $(t_n,s_n) \to (t,s) \in [0,1]\times [0,r]$,
\begin{align*}
\gamma_n \to \gamma \text{ in } C^{2}(S^{1},S^{2}),  
\end{align*}
where $|\dot\gamma|_{g_t}\equiv {\rm const}$ and
\begin{align}
\label{eq:1}
c_{g_t}(\gamma,\theta)= s c(\gamma(\theta))\; \forall \theta \in S^{1}.  
\end{align}
Consequently, if $\mathcal{M}_r$ is not compact, there is $(t,s)\in [0,1]\times [0,r]$ and 
$\gamma_r \in C^{2}(S^{1},S^{2})$
satisfying $|\dot\gamma|_{g_t}\equiv {\rm const}$ and (\ref{eq:1}),
which is not embedded, but a limit of embedded curves in $\mathcal{M}_r$.
Thus there are $\theta_1,\theta_2 \in S^{1}$, 
such that $\theta_1\neq \theta_2$ and $\gamma_r(\theta_1)= \gamma_r(\theta_2)$.
From (\ref{eq:2}) we deduce that there is $\delta>0$ independent of $r$, such that
\begin{align}
\label{eq:3}
\delta \le |\theta_1-\theta_2|\le  1-\delta.  
\end{align}
Hence for any $n \in \nz$ there is $\gamma_n \in \mathcal{M}_r$ such that
\begin{align}
\label{eq:4}
\dist(\gamma_n(\theta_1),\gamma_n(\theta_2)) \le \frac{1}{n}.  
\end{align}
To obtain a contradiction assume there is $(r_n)$ converging to $0$ 
such that $\mathcal{M}_{r_n}$ is not compact. 
Then for any $n \in \nz$ there are $(t_n,s_n) \in [0,1]\times [0,r_n]$, 
$\theta_{1,n},\theta_{2,n} \in S^{1}$, and 
$\gamma_n \in \mathcal{M}_{r_n}$
that satisfy (\ref{eq:3}) and (\ref{eq:4}). 
From the uniform bounds, going to a subsequence, we may assume
that $(t_n,s_n,\gamma_n,\theta_{1,n},\theta_{2,n})$ converge to 
$(t,0,\gamma,\theta_1,\theta_2)$, where
$\theta_1$ and $\theta_2$ satisfy (\ref{eq:3}) and $\gamma$ is a closed nontrivial geodesic 
in $(S^2,g_t)$ satisfying 
$\gamma(\theta_1)=\gamma(\theta_2)$. This contradicts Lemma \ref{lem_geodesic_simple}.
Since all the above bounds are uniform with respect to $\|c\|_\infty$,
the constant $\eps_0>0$ may be chosen uniform with respect to $\|c\|_\infty$
as well.  
\end{proof}

\section{Existence results}
\label{sec:existence}
We follow \cite{arXiv:0808.4038} and consider solutions to \eqref{eq:1} as zeros of 
the vector field $X_{c,g}$ defined on the Sobolev space $H^{2,2}(S^1,S^2)$ as follows:
For $\gamma \in H^{2,2}(S^1,S^2)$ we let $X_{c,g}(\gamma)$ be the unique weak solution of   
\begin{align}
\label{eq:def_vector_field}
\big(-D_{t,g}^{2} + 1\big)X_{c,g}(\gamma)= 
-D_{t,g} \dot\gamma + |\dot \gamma|_{g} c(\gamma)J_{g}(\gamma)\dot\gamma   
\end{align}
in $T_\gamma H^{2,2}(S^1,S^2)$.\\
Solutions to \eqref{eq:1} or equivalently zeros of $X_{c,g}$  
are invariant under a circle action:
For $\theta \in S^{1}=\rz/\zz$ and $\gamma \in H^{2,2}(S^1,S^2)$
we define $\theta*\gamma \in H^{2,2}(S^1,S^2)$ by
\begin{align*}
\theta*\gamma(t) = \gamma(t+\theta).  
\end{align*}
Thus, any solution gives rise to a $S^{1}$-orbit of solutions and we say
that two solutions $\gamma_1$ and $\gamma_2$ are (geometrically) distinct, if
$S^{1}*\gamma_1 \neq S^{1}*\gamma_2$.\\
We denote by $M\subset H^{2,2}(S^1,S^2)$ the set
\begin{align*}
M:= \{\gamma \in H^{2,2}(S^1,S^2) 
\where \dot\gamma(\theta) \neq 0\, \forall \theta \in S^1 
\text{ and }\gamma \text{ is embedded.}\}.
\end{align*}
In \cite{arXiv:0808.4038} an integer valued $S^{1}$-degree, $\chi_{S^{1}}(X_{c,g},M)$  
is introduced.
The $S^{1}$-degree
is defined, whenever $X_{c,g}$ is proper in $M$, i.e.
the set $\{\gamma \in M \where X_{c,g}(\gamma)=0 \}$ is compact,
and does not change under homotopies in the class
of proper vector fields.

\begin{theorem}
Suppose $(S^2,g)$ has positive Gauss curvature.
Then there is $\eps_0>0$ such that for all smooth functions
$c:S^2 \to \rz$ satisfying $0< c \le \eps_0$
there are two embedded geometrically distinct closed curves 
which solve equation \eqref{eq:prescribed_geodesic}.   
\end{theorem}
\begin{proof}
From the uniformization theorem up to isometries we may assume
without loss of generality that
\begin{align*}
g= e^\phi g_{can}, 
\end{align*}
where $\phi \in C^\infty(S^2,\rz)$ and $g_{can}$ denotes the 
standard round metric on $S^2$.\\
We consider the set of metrics $\{g_t\where t\in [0,1]\}$ 
defined by
\begin{align*}
g_t &:= e^{t\phi} g_{can}.
\end{align*}
From Lemma \ref{l:proper} there is $\eps_0>0$
such that the set
\begin{align*}
\{\gamma \in M \where X_{c,g_{t}}(\gamma) = 0 \text{ for some }
t \in [0,1]\}  
\end{align*}
is compact for all functions $c$ with $0<c\le \eps_0$. Consequently, 
\begin{align*}
[0,1]\ni t\mapsto X_{c,g_t}  
\end{align*}
is is a homotopy of proper vector fields.
From \cite{arXiv:0808.4038} there holds
\begin{align*}
-2= \chi_{S^1}(X_{c,g_{can}},M),  
\end{align*}
such that the homotopy invariance leads to
\begin{align*}
\chi_{S^1}(X_{c,g},M)=-2  
\end{align*}
Since the local degree of an isolated zero orbit 
is greater than or equal to $-1$ by \cite[Lem 4.1]{arXiv:0808.4038}, there
are at least two geometrically distinct solutions to \eqref{eq:prescribed_geodesic}.
This gives the claim.  
\end{proof}

%***********************************************************************
%***********************************************************************
\bibliographystyle{abbrv}
\bibliography{simple_convex_curve}

\end{document}